\documentclass{amsart}
\usepackage[all]{xy}
\usepackage[normalem]{ulem}
\usepackage{amssymb}%
\usepackage{graphicx}%
\usepackage{amsthm}%
\usepackage{amsmath}%
\usepackage{amsfonts} 
\usepackage{latexsym}%
\usepackage{epsfig}%
\usepackage{epic}%
\usepackage{amscd}%
\usepackage{xcolor}%
\usepackage[linktoc=all, pagebackref, hyperindex]{hyperref}%

\theoremstyle{plain}
  \newtheorem{thm}{Theorem}[section]
  \newtheorem{prop}[thm]{Proposition}
  \newtheorem{lem}[thm]{Lemma}
  \newtheorem{cor}[thm]{Corollary}
  
\theoremstyle{definition}
  \newtheorem{dfn}[thm]{Definition}
  \newtheorem{exmp}[thm]{Example}
  \newtheorem{exmps}[thm]{Examples}

  \newtheorem{observ}[thm]{Observation}
\theoremstyle{remark}
  \newtheorem{rem}{Remark}

\newcommand{\lk}{\mathrm{link}}
\newcommand{\del}{\mathrm{del}}
\newcommand{\ass}{\mathrm{ass}}
\newcommand{\Ass}{\mathrm{Ass}}
\newcommand{\fp}{\mathfrak{p}}
\newcommand{\fq}{\mathfrak{q}}
\newcommand{\depth}{\mathrm{depth}}
\newcommand{\sdepth}{\mathrm{sdepth}}
\newcommand{\ZZ}{\mathbb{Z}}
\newcommand{\KK}{\mathbb{K}}
\newcommand{\cF}{\mathcal{F}}
\newcommand{\x}{\textbf{x}}
\newcommand{\bigh}{\mathrm{bigh}}
\newcommand{\height}{\mathrm{ht}}
\newcommand{\reg}{\mathrm{reg}}
\newcommand{\cG}{\mathcal{G}}

\newcommand{\excise}[1]{}

\title{Shedding vertices and ass-decomposable monomial ideals}

\author[R. Jafari]{Raheleh Jafari}
\address{Mosaheb Institute of Mathematics, Kharazmi Universtity, Tehran, Iran.}
\email{rjafari@ipm.ir}

\author[A. A. Yazdan Pour]{{Ali Akbar} {Yazdan Pour}}
\address{Department of Mathematics, Institute for Advanced Studies in Basic Sciences (IASBS), 	P.O.Box 45195-1159, Zanjan, Iran.}
\email{yazdan@iasbs.ac.ir}

\date{March 8, 2021}

\keywords{shedding vertex, ass-decomposable ideal, sequentially Cohen-Macaulay, Stanley inequality}
\subjclass[2010]{13F55, 05E45}

\begin{document}
\maketitle

\begin{abstract}
The shedding vertices of simplicial complexes are studied from an algebraic point of view. Based on this perspective, we introduce the class of ass-decomposable monomial ideals which is a generalization of the class of Stanley-Reisner ideals of vertex decomposable simplicial complexes. The recursive structure of ass-decomposable monomial ideals allows us to find a simple formula for the depth, and in squarefree case, an upper bound for the regularity of such ideals.
\end{abstract}

\section*{Introduction}~\label{intro}
The concept of a (pure) vertex decomposable simplicial complex was introduced by Provan and
Billera \cite{Provan-Billera} in order to find an upper bound for the diameter of a polyhedron in terms of the number of its facets. Later Bj\"orner and Wachs extended this notion to non-pure simplicial complexes \cite{Bjorner}. In \cite{Wachs}, Wachs shows that a vertex decomposable
simplicial complex is shellable and hence is homotopy equivalent to a wedge of spheres \cite[Theorem 4.1]{Bjorner}. Recently Coleman et. al \cite{Coleman} show that if $\Delta$ is a vertex decomposable complex on the vertex set $V$, then there exists an ordering of its ground set $V$ such that adding the revlex smallest missing $(d+1)$-subset of $V$, results in a simplicial complex that is again vertex decomposable. This shows that vertex decomposable complexes are shelling completable and hence satisfy the Simon's Conjecture (see \cite{Coleman} for required definition).  Inspired by the importance of vertex decomposable complexes, in this paper, we introduce a class of monomial ideals called ``ass-decomposable" monomial ideals which provides a generalization of the Stanley-Reisner ideal of vertex decomposable simplicial complexes. Our aim in this direction is to study the algebraic invariants of such ideals. One of the aspects that makes the vertex decomposable simplicial complexes important is that the Stanley-Reisner ring of  a vertex decomposable simplicial complex $\Delta$ is sequentially Cohen-Macaulay. This observation, among other things, implies that $\depth(R/I_\Delta)= \min\{\dim R/\fp \colon \; \fp \in \ass(I_\Delta)\}$ (see observation~\ref{observation}). The natural question is that for what other classes of monomial ideal the depth formula can be obtained as above. We show that if $I$ is ass-decomposable then depth of $R/I$ is the minimum of $\dim R/\fp$ where $\fp \in \ass(I)$ (Theorem~\ref{thm:depth}), though $R/I$ is not sequentially Cohen-Macaulay in general (Example~\ref{ass-decomposable ideal satisfies in stanley conjecture}). In order to introduce the class of ass-decomposable monomial ideal,  first we study shedding vertices of vertex decomposable simplicial complexes from algebraic point of view. Let $I_\Delta$ denote the Stanley-Reisner ideal of a simplicial complex $\Delta$.  It turns out that $v$ is a shedding vertex of $\Delta$ precisely when the minimal prime ideals of $I_\Delta+(x_v)$ are those minimal prime ideals of $I_\Delta$ that contain $x_v$, see Proposition~\ref{controlling associated primes}. 
In view of this characterization of shedding vertices, one of the main results of Section~\ref{sec: shedding vertices} asserts that $\Delta$ is sequentially Cohen-Macaulay, if $\Delta$ has a shedding vertex such that both $\lk_{\Delta}(v)$ and $\del_\Delta(v)$ are sequentially Cohen-Macaulay (cf. Theorem~\ref{seqentially of link and del implies Delta itself}). This gives a direct proof for the well-known implication 
\begin{center}
vertex decomposable $\Longrightarrow$ sequentially Cohen-Macaulay.
\end{center}

Motivated by Proposition~\ref{controlling associated primes}, in Section \ref{sec 2}, we introduce ass-decomposable monomial ideals (cf. Definition~\ref{definition of ass-decomposable}). The squarefree ass-decomposable monomial ideals generated in degree greater than $1$, are nothing but the Stanley-Reisner ideals of vertex decomposable simplicial complexes (cf.~Corollary~\ref{squarefree decomposable}). 
The ass-decomposability, somehow, comes from comparing the associated prime ideals. In this regard, the unmixed property  is very effective for $I$. Indeed ass-decomposability is preserved under taking radical for unmixed ideals (cf. Lemma~\ref{sep-rad}). As an immediate consequence of the  depth formula for ass-decomposable monomial ideal, we derive that  unmixed ass-decomposable monomial ideals are Cohen-Macaulay (cf. Corollary~\ref{unmixed ass-decomposable is CM}).

Recall that a finitely generated  $\ZZ^n$-graded module $M$ is said to satisfy the \textit{Stanley's inequality}, if $\sdepth(M)\geq\depth(M)$, where $\sdepth(M)$ denote the Stanley
depth of $M$.  In fact, Stanley \cite{Stanley-1982} conjectured that all  $\ZZ^n$-graded modules over $n$-dimensional polynomial rings satisfy the Stanley's inequality. This conjecture has been disproved in \cite{Duval}. However, it is still interesting to find  some classes of modules which satisfy Stanley's inequality. For a survey on this topic, we refer the reader to \cite{survey}. The unmixed ass-decomposable monomial ideals satisfy the  Stanley's inequality (cf. Corollary~\ref{ass-decomposable ideal satisfies in stanley conjecture}).

The last part of Section~\ref{sec 2}  is devoted to find an upper bound for the Castelnuovo-Mumford regularity of squarefree ass-decomposable monomial ideals. We show that the regularity of a squarefree ass-decomposable monomial ideal is at most the index of irreducibility of the ideal and we characterize when this upper  bound holds.

\section{Preliminaries}
Throughout this paper $R=\mathbb{K}[x_1,\ldots, x_n]$ denotes the polynomial ring in $n$ variables $x_1, \ldots,x_n$ over a field $\mathbb{K}$ endowed with the standard grading, that is  $\deg (x_i)=1$ and $\mathfrak{m}$ denotes the irrelevant homogeneous maximal ideal of $R$. 
\subsection{Algebraic background}

Let $H_\mathfrak{m}^i(\mbox{-})$ be the $i$th local cohomology functor with support in $\mathfrak{m}$, by which we mean the $i$th right derived functor of $\Gamma_\mathfrak{m}(\mbox{-})$. Recall that $\Gamma_{\mathfrak{m}}(M)=\cup_{n=0}^\infty(0 \mathop{\colon}\limits_M \mathfrak{m}^n)$, for any $R$-module $M$. It turns out that 
\[H_{\mathfrak{m}}^i(M) = \mathop{\lim_{\longrightarrow}}_{n \in \mathbb{N}} \;\mathrm{Ext}^i_R(R/\mathfrak{m}^n, M).\]
The depth and the (Krull) dimension of a finitely generated $R$-module $M$,  can be recognized by the first and the last non-vanishing local cohomology module of $M$, respectively. More precisely,
\begin{align*}
\depth(M)=\min\{i \colon H_{\mathfrak{m}}^i(M) \neq 0\},\\
\dim (M)= \max\{i \colon H_{\mathfrak{m}}^i(M) \neq 0\}.
\end{align*}
The reader is refereedto \cite{Brodmann-Sharp} for more details about local cohomology.

We are interested in interpretation of algebraic properties in terms of associated prime ideals. Recall that an ideal $I\subset R$ has a minimal  primary decomposition $I=\cap^r_{i=1}\fq_i$,
where each $\fq_i$ is a $\fp_i$-primary ideal. Let $\ass(I)=\{\fp_1,\dots,\fp_r\}$ denote 	the  set of associated prime ideals of $R/I$. An ideal $I \subseteq R$ is called \textit{unmixed} if $\dim R/I = \dim R/\fp$ for all $\fp \in \ass(I)$.

For a Cohen-Macaulay $R$-module $M$, that is a finitely generated $R$-module with $\depth (M)=\dim(M)$, we have 
\begin{equation}\label{1}
\depth(M)=\dim R/\fp, \quad \text{ for all } \fp\in\Ass(M)
\end{equation}
(see \cite[Theorem 2.1.2]{BH}).

A related concept is the notion of sequentially Cohen-Macaulay modules. A finitely generated (graded) $R$-module $M$ is called \textit{sequentially Cohen-Macaulay}, if there exists a filtration
\begin{equation} \label{filtration}
0 =M_0 \subset M_1 \subset \cdots \subset M_r =M
\end{equation}
of (graded) submodules of $M$ such that each quotient $M_i/M_{i-1}$ is Cohen-Macaulay and
\[\dim M_1/M_0 < \dim M_2/M_1 < \cdots < \dim M_r/M_{r-1}.\]

The following observation yields a simple method for computing the depth of a sequentially Cohen-Macaulay module and will be used in Theorem~\ref{seqentially of link and del implies Delta itself}.

\begin{observ}\label{observation}
Let $M$ be a finitely generated $d$-dimensional sequentially Cohen-Macaulay $R$-module. By \cite[Lemma 5.4]{Schenzel}, one has
\begin{equation} \label{local cohomology of SCM modules}
H_{\mathfrak{m}}^i \left( M \right) = H_{\mathfrak{m}}^i \left( D_i (M) \right) = H_{\mathfrak{m}}^i \left( D_i (M)/ D_{i-1}(M) \right),
\end{equation}
for all $i=0, \ldots, d$, where $D_i(M)$ denotes the largest submodule of $M$ such that $\dim D_i(M) \leq i$. Note that $D_i (M)/ D_{i-1}(M)$ are either zero or $i$-dimensional Cohen-Macaulay $R$-modules, by \cite[Proposition~1.1 and Corollary~2.7]{Herzog-Popescu}. Since $\Ass(M)= \bigcup_{i=1}^d \Ass(D_i (M)/ D_{i-1}(M))$ (\cite[Corollary 2.3]{Schenzel}), as a consequence of (\ref{1}) and (\ref{local cohomology of SCM modules}), we have

\begin{equation} \label{depth of SCM modules}
\begin{split}
\depth (M) & = \min \{i \colon \; H_{\mathfrak{m}}^i \left( M \right) \neq 0 \}  \\
&=\min\{i \colon H_{\mathfrak{m}}^i(D_i (M)/ D_{i-1}(M)) \neq 0\}\\
&=\min\{\depth\, (D_i (M)/ D_{i-1}(M)) \colon \; D_i (M)/ D_{i-1}(M) \neq 0\}\\
& = \min \{\dim R/\fp \colon \; \fp \in \Ass(M) \}.
\end{split}
\end{equation}
\end{observ}
	
Let  $I$ be an ideal in the polynomial ring $R$ and $\bigh(I)=\max\{\height(\fp) \colon \; \fp \in  \ass(I)\}$ denotes the big height of $I$. Then  the above observation shows that $\depth(R/I)=n-\bigh(I)$, if $R/I$ is sequentially Cohen-Macaulay.

\subsection{Simplicial complexes and Stanley-Reisner ring}
A \textit{simplicial complex} $\Delta$  on the vertex set $V=V(\Delta)$ is a collection of subsets of $V$, called \textit{faces} of $\Delta$, such that (i) $\{v\} \in \Delta$, for all $v \in V$, (ii) if $F\in\Delta$ and  $G\subseteq F$, then $G\in\Delta$.  The set of maximal faces of $\Delta$ with respect to inclusion, which are called \textit{facets}, is denoted by $\cF(\Delta)$. If $\cF(\Delta) = \{F_1, \ldots, F_m\}$, we often write $\Delta=\langle F_1, \ldots, F_m \rangle$ and $\Delta$ is called a \textit{simplex}, if $\cF(\Delta)$ is a singleton. Also the $i$th pure skeleton of $\Delta$ is the simplicial complex $\Delta^{[i]}= \langle F \in \Delta \colon \; |F|=i \rangle$.

All over this paper, $\Delta$ is  a simplicial complex on $V=[n]$. For each subset $F\subseteq[n]$ we set
\[\x_F=\prod_{i\in F}x_i.\]
The \textit{Stanley-Reisner ideal} of $\Delta$ is the ideal $I_\Delta$ of $R$, which is generated
by those squarefree monomials $\x_F$ with $F\notin\Delta$. A simplicial complex $\Delta$ is called \textit{Cohen-Macaulay} if the quotient ring $\mathbb{K}[\Delta]=R/I_\Delta$ is Cohen-Macaulay and $\Delta$ is called \textit{sequentially Cohen-Macaulay}, if all pure $i$-skeletons of $\Delta$ are Cohen-Macaulay. It turns out that $\Delta$ is sequentially Cohen-Macaulay if and only if the quotient ring $R/I_\Delta$ is sequentially Cohen-Macaulay \cite[Theorem 3.3]{Duval 2}.

For a subset $F\subseteq[n]$, let $\bar{F}=[n]\setminus F$ and $P_{\bar{F}}$ be the prime ideal $(x_i \colon \; i\in [n] \setminus F)$. It follows from \cite[Theorem 5.1.4]{BH} that
\begin{equation*} \label{prime decomposition of I_Delta}
I_\Delta=\bigcap_{F\in\cF(\Delta)}P_{\bar{F}},
\end{equation*}
is a minimal prime decomposition of $I_\Delta$.

\subsection{Graphs and edge ideals}
A simple graph is a pair $G=(V,E)$ where $V$ is a set whose elements are called \textit{vertices} and $E$ is a collection of $2$-subsets of $V$. Any element of $E$ is called an \textit{edge} of $G$. All graphs in this paper are supposed to be finite graph in the sense that $|V|<\infty$. The \textit{edge ideal} of a graph $G=(V,E)$ is the quadratic square free monomial ideal
\[I(G)=(x_vx_w \colon \{v,w\} \in E),\]
considered as an ideal in the polynomial ring $\mathbb{K}[x_v \colon v \in V]$. For any graph $G$, there are two typical instances of simplicial complexes arising from $G$, namely the independence complex and clique complex. A subset $A \subseteq [n]$ is called an \textit{independent set} if $e \nsubseteq A$, for all $e \in E$. The collection of independent sets in $G$ forms a simplicial complex $\Delta_G$ which is called the \textit{independence complex} of $G$. One may easily check that $I_{\Delta_G}=I(G)$. Alternatively, a subset $A \subseteq [n]$ is called a \textit{clique} in $G$ if $\{v,w\} \in E$ for all $v,w \in A$ with $v \neq w$. The collection of cliques in $G$ forms a simplicial complex $\Delta(G)$ which is called the \textit{clique complex} of $G$.

Chordal graphs have been studied extensively in combinatorial commutative algebra. A \textit{cycle} of length $r$ in a graph $G=(V,E)$ is a sequence $v=v_1,v_2,\ldots,v_{r-1},v_r=v$ of vertices of $G$ such that $\{v_i, v_{i+1}\} \in E$ and a graph $G$ is \textit{chordal} if each cycle of length four or more has a chord, by which we mean an edge joining two vertices that are not adjacent in the cycle. It is known that $\Delta_G$ is sequentially Cohen-Macaulay, if $G$ is a chordal graph \cite[Theorem 3.2]{Francisco-Van Tuyl}. Trivial examples of chordal graphs are \textit{forests}, the graphs without any cycle. 

Let $G=(V,E)$ be a finite simple graph. If $v$ is a vertex of $G$, the \textit{neighborhood} of $v$ is the set $N_G(v)=\{w \in V\colon \; \{v,w\} \in E\}$ while the \textit{closed neighborhood} of $v$ is the set $N_G[v]=\{v\} \cup N_G(v)$. The number $\deg(v)=|N_G(v)|$ is called the \textit{degree} of $v$ and $v$ is called a \textit{pendant} vertex if $\deg(v)=1$. Note that every forest has a pendant vertex. Finally, by graph $G\setminus v$, we mean the graph $G\setminus v=(V\setminus\{v\},E')$ where $E'=\{e \in E \colon \; v \notin e\}$ and $G\setminus A$ stands for the graph $G\setminus v_1 \setminus \cdots \setminus v_r$, where $A=\{v_1, \ldots, v_r \} \subseteq V$.

\subsection{Vertex decompasable simplicial complexes and shellability}
Vertex decomposable simplicial complexes form a well-behaved class of simplicial complexes and have been studied in several literatures (\cite{Biermann1, Biermann 2, Cook-Nagel, Dochtermann, KM, MK, Moradi-Kiani, RY, Van Tuyl, Woodroofe, Woodroofe2}). This class of simplicial complexes are defined recursively in terms of shedding vertices. Recall that the \textit{link} of a face $F$ in $\Delta$ is defined as
\[\lk_{\Delta}(F)=\{G\in\Delta \colon \;  G\cap F=\varnothing, \, G\cup F\in\Delta\},\]
and the \textit{deletion} of $F$ is the simplicial complex
\[\del_{\Delta}(F)=\{G\in\Delta \colon \;  G \cap F = \varnothing\}.\]
A vertex $v \in V(\Delta)$ is called a \textit{shedding} vertex of $\Delta$ if any facet of $\del_\Delta(v)$ is a facet of $\Delta$. It turns out that $v$ is a shedding vertex of $\Delta$ if and only if no face of $\lk_\Delta\left( v \right)$ is a facet of $\del_\Delta(v)$. A simplicial complex $\Delta$ is called \textit{vertex decomposable}, if either it is a simplex or else there
is a shedding vertex $v\in V(\Delta)$ such that both $\lk_{\Delta}(v)$ and $\del_\Delta(v)$ are vertex decomposable.

\begin{exmps}
\mbox{}
\begin{itemize}
\item[(i)] Let $\Delta=\langle\{1,2,3\},\{2,3,4\}\rangle$. Then $\lk_\Delta(1)=\langle\{2,3\}\rangle$ and $\del_\Delta(1)=\langle\{2,3,4\}\rangle$ are simplexes. As $\{2,3,4\}\in\cF(\Delta)$ is the only facet of $\del_\Delta(1)$, the vertex $1$ is a shedding vertex of $\Delta$.  Therefore, $\Delta$ is vertex decomposable.
\item[(ii)] Let $G$ be a forest and $v$ be a pendant vertex of $G$. Let $w$ be a neighbor of $v$ in $G$. One can easily check that $w$ is a shedding vertex of $\Delta:=\Delta_G$. Moreover we have
\begin{align*}
\lk_\Delta(w)&= \Delta_{G\setminus N_G[w]},\\
\del_\Delta(w)&=\Delta_{G\setminus w}
\end{align*}
are the independence complexes of some forests. Using induction, both $\lk_\Delta(w)$ and $\del_\Delta(w)$ are vertex decomposable. Hence $\Delta$ is vertex decomposable.
\end{itemize}
\end{exmps}

A simplicial complex $\Delta$ is called \textit{shellable} if there is an ordering $F_1, \ldots, F_m$ of the facets of $\Delta$ such that the intersection $\langle F_i \rangle \cap \langle F_1, \ldots, F_{i-1} \rangle$ is generated by some maximal proper subsets of $F_i$, for $i=1, \ldots, m$. In \cite[Lemma 6]{Wachs} it is shown that if $\Delta$ has a shedding vertex such that both $\lk_{\Delta} (v)$ and $\del_\Delta(v)$ are shellable, then $\Delta$ is shellable. It turns out that the vertex decomposable simplicial complexes are shellable and hence the Stanley-Reisner ring of such simplicial complexes are sequentially Cohen-Macaulay \cite[Page 87]{Stanley-book}. Based on the above arguments we have the following well-known implications which both of them are known to be strict.
\begin{center}
vertex decomposable $\Longrightarrow$ shellable $\Longrightarrow$ sequentially Cohen-Macaulay.
\end{center}
The following example summarizes some known classes of vertex decomposable simplicial complexes.

\begin{exmps}[Some known classes of vertex decomposable simplicial complexes]
\mbox{}
\begin{itemize}
\item[(i)] The independence complex of a chordal graph is vertex decomposable (see {\cite[Corollary~7]{Woodroofe} or \cite[Theorem 4.1]{Dochtermann}}). More generally, If $G$ is a graph with no chordless cycles of length other than 3 or 5, then $\Delta_G$ is vertex decomposable \cite[Theorem 1]{Woodroofe}.
\item[(ii)] If $\Delta$ is a simplicial complex on the vertex set $[n]$, the \textit{Alexander dual} $\Delta^\vee$ of $\Delta$ is a simplicial complex defined as follows:
\[\Delta^\vee =\{[n]\setminus F \colon F \notin \Delta\}.\]
If $G$ is a chordal graph then $\Delta(G)^\vee$ is vertex decomposable  \cite[Corollary 2.6]{FM}. 
\item[(iii)] Let $\Delta$ be a vertex decomposable simplicial complex. Then 
\begin{itemize}
\item[(a)] the $i$th skeleton $\Delta^{(i)}$ of $\Delta$ defined as $\Delta^{(i)}:= \{ F \in \Delta \colon\; \dim F \leq i\}$ is vertex decomposable for all $i$ \cite[Lemma 3.10]{Woodroofe2}. In particular all (pure) skeleton of a simplex is vertex decomposable.
\item[(b)] $\lk_\Delta(F)$ is vertex decomposable, for any $F \in \Delta$ \cite[proposition 2.3]{Provan-Billera}.
\item[(c)] If $\Delta'$ is another vertex decomposable simplicial complex, then $\Delta \star \Delta'=\{ F \cup F' \colon \; F \in \Delta, \; F' \in \Delta'\}$ is also vertex decomposable \cite[Proposition 2.4]{Provan-Billera}.
\end{itemize}
\item[(iv)] Every simplicial complex whose the geometric realization is homeomorphic to $2$-ball or $2$-sphere is vertex decomposable \cite[Theorem 3.1.3]{Provan-Billera}.
\item[(v)] If $\Delta$ is a $d$-dimensional shellable simplicial complex on $d + 3$ vertices, then
$\Delta$ is vertex decomposable \cite[Theorem 4.4]{Coleman}.
\end{itemize}
\end{exmps}

\section{Shedding vertices of simplicial complexes} \label{sec: shedding vertices}
In this Section, we study the shedding vertices of a simplicial complex from algebraic point of view. Indeed the following proposition yields an algebraic counterpart of shedding vertices and will be useful in Section~\ref{sec 2} for introducing a class of monomial ideals as a generalization of Stanley-Reisner ideals of vertex decomposable simplicial complexes for non-squarefree cases.

\begin{prop}\label{controlling associated primes}
The following statements are equivalent for a vertex $v$ of a simplicial complex $\Delta$.
\begin{itemize}
\item[\rm (i)] $v$ is a shedding vertex of $\Delta$;
\item[\rm (ii)] $\mathrm{ass}\left( I_{\Delta}, x_v \right)= \left\{ \mathfrak{p} \in \mathrm{ass}\left( I_{\Delta} \right) \colon \; x_v \in \mathfrak{p} \right\}$. 
\end{itemize}
\end{prop}

\begin{proof}
	Let  $I_\Delta=\cap_{i=1}^r \mathfrak{p}_i$ be  a minimal prime decomposition of $I_\Delta$. We may assume that $x_v \in \cap^s_{i=1} \mathfrak{p}_i$ and $x_v \notin \cup^r_{i=s+1} \mathfrak{p}_i$, for some $1 \leq s \leq r$. For all $1 \leq i \leq r$, let $\fp_i=P_{\bar{F}_i}$ where $F_i\in\Delta$.
	
	(i)$\implies$(ii) Fix $s+1 \leq i \leq r$. Since $v \in F_i$ is a shedding vertex of $\Delta$, we conclude that $F_i \setminus \{v\} \in \mathrm{link}_\Delta\left( v \right)$ is not a facet of $\mathrm{del}_{\Delta} (v)$. Hence there exists $G \in \mathcal{F} \left( \mathrm{del}_{\Delta} (v) \right)$ such that $F_i \setminus \{v\} \subset G$. Indeed, $G \in \mathcal{F} \left( \Delta \right)$ and $v \notin G$. Thus $G = F_t$, for some $1 \leq t \leq s$. Moreover, $\bar{F}_t = \bar{G} \subseteq \bar{F_i} \cup \{v\}$ and consequently $\mathfrak{p}_t = \mathfrak{p}_{\bar{F}_t} \subseteq\left(\mathfrak{p}_i, x_v \right)$. This shows that
	\[
	\left(I_\Delta,x_v \right) = \left( \mathop{\cap}^s_{j=1} \mathfrak{p}_j \right) \cap \left( \mathop{\cap}^r_{i=s+1} \left( \mathfrak{p}_i,x_v \right) \right) = \mathop{\cap}^s_{j=1} \mathfrak{p}_j,
	\]
is a minimal prime decomposition of $I_\Delta+(x_v)$.

	(ii)$\implies$(i) Our hypothesis implies that for all $s+1 \leq i \leq r$, there exists $1 \leq j_i \leq s$ such that $\mathfrak{p}_{j_i} \subseteq \left( \mathfrak{p}_i, x_v \right)$. Translating this in terms of facets of $\Delta$, we conclude that no facet of $\mathrm{link}_\Delta\left( v \right)$ is a facet of $\del_{\Delta}(v)$. 
\end{proof}

For a vertex $v$ of $\Delta$, let $I_\Delta(v)=(x_w \colon\; x_vx_w\in I_\Delta )$. The following remark enables us to find the minimal primary decomposition of $I_\Delta \colon x_v$ and $I_\Delta+(x_v)$ in terms of those of $I_\Delta$, when $v$ is a shedding vertex of $\Delta$.

\begin{rem}\label{rem-2}
Let $v$ be a shedding vertex of  $\Delta$ and $I_\Delta=\cap^r_{i=1}\fp_i$ be a minimal prime decomposition of $I_\Delta$. Assume without loss of generality that $x_v\in\cap^s_{i=1}\fp_i$ and $x_v\notin\cup^r_{i=s+1}\fp_i$ for some $1\leq s\leq r$. Let $\fp_i=P_{\bar{F_i}}$ for $F_i\in\cF(\Delta)$ and $i=1,\dots,r$. Then $v\notin \cup^s_{i=1}F_i$ and $v\in\cap^r_{i=s+1}F_i$. For $i=s+1,\dots,r$, let  $Q_i=F_i\setminus\{v\}$. Then 	
\begin{enumerate}
\item $\del_\Delta(v)= \langle F_1,\dots,F_s \rangle$.
\item $\lk_\Delta(v)=\langle Q_{s+1},\dots,Q_r \rangle$.
\item $(I_\Delta,x_v)=\cap^s_{i=1}\fp_i=I_{\del_\Delta(v)}+(x_v)$.
\item $I_\Delta \colon x_v=\cap^r_{i=s+1}\fp_i=I_{\lk_\Delta(v)}+I_\Delta(v)$.
\item  If $\Delta$ is vertex decomposable, then $\lk_\Delta(w)$ is also vertex decomposable, for any vertex $w$ (cf. \cite[Proposition~3.7]{Woodroofe2}).
\end{enumerate}
\end{rem}

It is known that if $\Delta$ has a shedding vertex such that both $\lk_{\Delta}(v)$ and $\del_\Delta(v)$ are shellable, then $\Delta$ is shellable (\cite[Lemma 6]{Wachs}). Note that the class of shellable simplicial complexes is a subclass of sequentially Cohen-Macaulay complexes, so it is natural to ask if the result holds replacing shellable simplicial complexes with sequentially Cohen-Macaulay complexes. Our algebraic characterization of shedding vertices enables us to obtain the later result for sequentially Cohen-Macaulay complex as well (cf. Theorem~\ref{seqentially of link and del implies Delta itself}). 
The following theorem shows the importance of shedding vertices in inductive processes for squarefree monomial ideals.

\begin{thm} \label{seqentially of link and del implies Delta itself}
Let $\Delta$ be a simplicial complex and $I=I_\Delta$ be the Stanley-Reisner ideal of $\Delta$. If $\Delta$ has a shedding vertex $v$ such that $\mathrm{link}_\Delta(v)$ and $\del_\Delta (v)$ are sequentially Cohen-Macaulay/Cohen-Macaulay/shellable, then
\begin{itemize}
\item[\rm (i)] $\depth \left( R/I \right) = \min \left\{ \depth \left( R/ \left(I \colon x_v \right) \right), \, \depth\left( R/ \left( I,x_v \right) \right) \right\}= n-\bigh(I)$;
\item[\rm (ii)] $\Delta$ is sequentially Cohen-Macaulay/Cohen-Macaulay/shellable.
\end{itemize}
\end{thm}

\begin{proof}
(i) It is well-known that shellable simplicial complexes or Cohen-Macaulay complexes are sequentially Cohen-Macaulay. So it is enough to prove (i) for sequentially Cohen-Macaulay complexes.
We note that $\left(I, x_v \right) = (x_v)+I_{\del_\Delta (v)}$  is sequentially Cohen-Macaulay complex, by our assumption. Since $\lk_\Delta (v)$ is sequentially Cohen-Macaulay, we derive that  the ideal $(I\colon x_v)$ is sequentially Cohen-Macaulay (see Remark~\ref{rem-2}(4)). Hence by \eqref{depth of SCM modules}, we conclude that
\begin{equation} \label{eq 1}
\begin{split}
 \depth \left( R/(I,x_v) \right) & = \min \{ \dim \left( R/\fp \right)\colon \; \fp\in\ass\left( \left( I, x_v \right) \right)\}, \\
 \depth \left( R/(I \colon x_v) \right) & = \min \{ \dim \left( R/\fp \right)\colon \; \fp\in\ass\left( \left( I \colon x_v \right) \right)\}.
\end{split}
\end{equation}
It follows from Proposition~\ref{controlling associated primes} that $\ass \left( I, x_v \right) \subset \ass(I)$. Clearly, $\ass \left( I \colon x_v \right) \subset \ass(I)$. So that $\ass(I) = \ass \left( I, x_v \right) \cup \ass \left( I \colon x_v \right)$. By \cite[Proposition 1.2.13]{BH}, we have $\depth (R/I) \leq \dim R/\mathfrak{p}$, for all $\mathfrak{p} \in \ass(I)$. Using \eqref{eq 1}, we conclude that 
\begin{equation}\label{eq 2}
\depth \left( R/I \right) \leq  \min \left\{ \depth \left( R/ \left(I \colon x_v \right) \right), \, \depth\left( R/ \left( I,x_v \right) \right) \right\}.
\end{equation}
On the other hand, from the short exact sequence
\begin{equation*} \label{frequenty used short exact sequence}
0\longrightarrow R/(I \colon x_v) \longrightarrow R/I \longrightarrow R/\left( I,x_v \right) \longrightarrow 0,
\end{equation*}
we obtain 
\begin{equation}\label{eq 3}
\depth \left( R/I \right) \geq  \min \left\{ \depth \left( R/ \left(I \colon x_v \right) \right), \, \depth\left( R/ \left( I,x_v \right) \right) \right\}.
\end{equation}
The inequalities \eqref{eq 2} and \eqref{eq 3} yield 
\begin{align*}
\depth \left( R/I \right) &= \min \left\{ \depth \left( R/ \left(I \colon x_v \right) \right), \, \depth\left( R/ \left( I,x_v \right) \right) \right\}\\
&=\min\{\dim (R/\fp) \colon \; \fp \in \ass(I)\} = n-\bigh(I).
\end{align*}

(ii) First we show that $\Delta$ is sequentially Cohen-Macaulay provided that $\mathrm{link}_\Delta(v)$ and $\del_{\Delta}(v)$ are sequentially Cohen-Macaulay. Let $\Delta^{\left[ i \right]}$ denotes the pure $i$-skeleton of the simplicial complex $\Delta$ and $I_i = I_{\Delta^{\left[ i \right]}}$. One may easily check that $\mathrm{link}_{\Delta^{\left[ i \right]}} (v) = \left( \mathrm{link}_{\Delta} (v) \right)^{\left[ i-1 \right]}$ and $\mathrm{del}_{\Delta^{\left[ i \right]}} (v) = \left( \mathrm{del}_{\Delta} (v) \right)^{\left[ i \right]}$. So that $\mathrm{link}_{\Delta^{\left[ i \right]}} (v)$ and $\mathrm{del}_{\Delta^{\left[ i \right]}} (v)$ are both Cohen-Macaulay, by our assumption. If $\Delta^{\left[ i \right]}$ is a simplex, then clearly $\Delta^{\left[ i \right]}$ is Cohen-Macaulay. So assume that $\Delta^{\left[ i \right]}$ is not a simplex. Then it is easily seen that $v$ is a shedding vertex in $\Delta^{\left[ i \right]}$.

Applying part (i) of the statement for the simplicial complex $\Delta^{\left[ i \right]}$, together with the fact that  $R/\left( I_i, x_v \right)$ and $R/\left( I_i\colon x_v \right)$ are Cohen-Macaulay, we find that $\Delta^{\left[ i \right]}$ is Cohen-Macaulay. Hence $\Delta$ is sequentially Cohen-Macaulay. In the case that $\mathrm{link}_\Delta(v)$ and $\del_{\Delta}(v)$ are Cohen-Macaulay, the same discussion as above shows that $\Delta$ is Cohen-Macaulay, while the shellable case follows from \cite[Lemma 6]{Wachs}.
\end{proof}

\section{ass-decomposable monomial ideals} \label{sec 2}
Several algebraic interpretations have been considered for the concept of vertex decomposability. A simplicial complex $\Delta$ is vertex decomposable if and only if $I_{\Delta^\vee}$ is $0$-decomposable in the sense of \cite{RY}, which holds precisely when $I_{\Delta^\vee}$ is a vertex splittable ideal introduced in \cite{MK}. In the following, we are looking for an algebraic property of $I_\Delta$ to characterize the vertex decomposability of $\Delta$. This property is not restricted to squarefree monomial ideals.

\begin{dfn}
\label{definition of ass-decomposable}
A monomial ideal $I$ is called \textit{ass-decomposable}, if either it is a   primary  ideal or else there exist a variable $x_i\notin\sqrt{I}$ and a positive integer $k$ such that
\begin{itemize}
\item[(i)] $\ass(I,x_i^k)=\{\fp\in\ass(I) \ ; \  x_i\in\fp\}$.
\item[(ii)] $(I,x_i^k)$ and $I \colon x_i^k$ are ass-decomposable. 
\end{itemize}
The monomial $x_i^k$ is called a \textit{decomposing monomial} of $I$.
\end{dfn}

\begin{rem}\label{k=1}
Let $I$ be a squarefree ass-decomposable monomial ideal with decomposing monomial $x_i^k$. Let $\fp$ be a minimal prime ideal of $I$ such that $x_i\notin\fp$. Then there exists a minimal prime ideal $\fq$ of $I$ such that $x_i\in\fq$ and $\fq\subseteq\fp+(x_i^k)$. Therefore, $k=1$. 
\end{rem}

As a consequence of Proposition~\ref{controlling associated primes} and Remarks~\ref{rem-2} and \ref{k=1}, we derive the following.
\begin{cor} \label{squarefree decomposable}
 $I_\Delta$ is ass-decomposable if and only if $\Delta$ is vertex decomposable. 
\end{cor}

\begin{exmp}\label{sep-not-seq}
Let $I=\cap^4_{i=1}\fq_i$, where $\fq_1=(x_1,x_2)$, $\fq_2=(x_3,x_1)$, $\fq_3=(x_1,x_2^2,x_3^2),\fq_4=(x_1,x_3^2,x_4), \fq_5=(x_3^2,x_4,x_5)$ in $R=\KK[x_1,\dots,x_5]$. Then 
\begin{eqnarray*}
I+(x_3^2)=\mathop{\cap}\limits^5_{i=2}\fq_i & ; & I\colon x_3^2=\fq_1\\
\left( \mathop{\cap}\limits^5_{i=2}\fq_i \right) +(x_4)=\fq_4\cap \fq_5 & ;& \left( \mathop{\cap}\limits^5_{i=2}\fq_i \right) \colon x_4=\fq_2\cap\fq_3\\
(\fq_2\cap\fq_3)+(x_2^2)=\fq_3 & ; & (\fq_2\cap\fq_3) \colon x_2^2=\fq_2\\
(\fq_4\cap\fq_5)+(x_5)=\fq_5 & ;& (\fq_4\cap\fq_5) \colon x_5=\fq_4
\end{eqnarray*}
Therefore $I$ is ass-decomposable.
\end{exmp}

Let $I \subseteq R$ be an ass-decomposable monomial ideal. Since the Stanley-Reisner ideal of a vertex decomposable simplicial complex is sequentially Cohen-Macaulay, it follows from Corollary~\ref{squarefree decomposable} that $R/I$ is sequentially Cohen-Macaulay, provided that $I$ is a squarefree ass-decomposable monomial ideal. The following example shows that this is not the case if $I$ is not squarefree. This example also shows that the radical of an ass-decomposable ideal need not to be an ass-decomposable ideal.

\begin{exmp}\label{ass-decomposable but not SCM}
	Let $I=\cap^3_{i=1}\fq_i$, where $\fq_1=(x_1,x_2)$, $\fq_2=(x_3,x_4)$ and $\fq_3=(x_1,x_3^2,x_4)$ in $R=\KK[x_1,x_2,x_3,x_4]$. Then 
	\begin{eqnarray*}
		I+(x_1)=\fq_1\cap\fq_3 & ; & I\colon x_1=\fq_2\\
		\left(\fq_1\mathop{\cap}\fq_3 \right)+(x_2)=\fq_1 & ;& \left( \fq_1\cap\fq_3 \right) \colon x_2=\fq_3
	\end{eqnarray*}
 Therefore $I$ is ass-decomposable. If $R/I$ is sequentially Cohen-Macaulay, then $R/(\fq_1\cap\fq_2)$ should be  Cohen-Macaulay by \cite[Lemma~2.21]{JS}, which is not the case because $\depth\, R/(\fq_1\cap\fq_2)=1 < 2 =\dim R/(\fq_1\cap\fq_2)$. Note that $\sqrt{I}= (x_1,x_2) \cap (x_3,x_4)$ is not  an ass-decomposable monomial ideal. 
\end{exmp}

In the above example, the quadratic monomial in $\fq_3$ allows to get ass-decomposability along with producing an associated prime ideal $\fp_3\supset\fp_2$. More precisely, $\sqrt{I}$ is not ass-decomposable, since $\fp_3$ is omitted from its associated prime ideals.
The following result shows that ass-decomposability of an unmixed monomial ideal implies the ass-decomposability of its radical.

\begin{lem}\label{sep-rad}
If $I$ is an unmixed ass-decomposable monomial ideal, then $\sqrt{I}$ is also ass-decomposable. 
\end{lem}  	
\begin{proof}
Let $I=\cap^r_{i=1}\fq_i$ be a minimal primary decomposition of $I$, where $\fq_i$ is $\fp_i$-primary. As $I$ is unmixed, $\fp_i \nsubseteq \fp_j$ for all $i \neq j$. It follows that $\sqrt{I}=\cap^r_{i=1}\fp_i$ is a minimal prime decomposition of $\sqrt{I}$. 

We use induction on $r=|\ass(I)|$ to show that $\sqrt{I}$ is ass-decomposable. If $r=1$, we have nothing to prove. So assume that $r>1$ and the assertion holds for any unmixed ass-decomposable monomial ideal whose number of irreducible components is less than $r$. Let the variable $x \notin \sqrt{I}$ and the positive number $k$ be as in the Definition~\ref{definition of ass-decomposable}. 
Without loss of generality, assume that $x \in \cap_{i=1}^t \fp_i \setminus \cup_{j=t+1}^r \fp_j$, for some $1 \leq t <r$. Since $\ass(I,x^k)=\{\fp_1, \ldots, \fp_t\}$, we conclude that $\ass (\sqrt{I},x)= \ass(\sqrt{(I,x^k)})=\{\fp_1, \ldots, \fp_t\}$. Note that $I+(x^k)$ and $I \colon x^k=\cap_{i=t+1}^r \fq_i$ satisfy in our induction hypothesis, so that $\sqrt{(I,x^k)}=(\sqrt{I},x)$ and $\sqrt{(I \colon x^k)}=\sqrt{I} \colon x$ are ass-decomposable. Thus $\sqrt{I}$ is ass-decomposable.
\end{proof}	

\begin{rem}
Let $I=\cap^r_{i=1}\fp_i$ be a minimal prime decomposition of  a squarefree monomial ideal $I$, with $r\geq2$. It's not difficult to observe that
	\[
	\fp_j^k\nsubseteq(x_s^t)+\fp_i^k \ \text{and} \ (x_s^t)+\fp_i^k \nsubseteq\fp_j^k,
	\]
for all $1\leq j\neq i\leq r$, $1\leq s\leq n$, $k\geq2$ and $t>1$.
Therefore, $\ass(I,x_s^t)$ has $r$ elements, which follows that the $k$th symbolic power of $I$, for $k\geq2$, is never ass-decomposable. In particular, the converse of Lemma~\ref{sep-rad} does not hold. 
\end{rem}

Let $I \subseteq R$ be an ass-decomposable monomial ideal. In spite of the fact that $R/I$ is not necessarily sequentially Cohen-Macaulay, but the following result shows the depth of $R/I$ can be computed by the same formula as in \eqref{depth of SCM modules}.

\begin{thm}\label{thm:depth}
 If $I$ is an ass-decomposable ideal, then
	 \[
 \depth(R/I)=\min\{\depth(R/\fp)\colon \; \fp\in\ass(I)\}=n-\bigh(I)
	 \]
\end{thm}
\begin{proof}
Let $I=\cap_{i=1}^r \fq_i$ be a minimal primary decomposition of the monomial ideal $I \subset R$. 	We proceed by induction on $r=|\ass(I)|$. If $r=1$, then $R/I$ is Cohen-Macaulay and there is nothing to prove. Assume that $r>1$ and the result is true for ass-decomposable monomial ideals with less than $r$ irreducible components. Consider the exact sequence
\begin{equation*}
0\longrightarrow R/(I\colon x^k)\longrightarrow R/I\longrightarrow R/(I,x^k)\longrightarrow0.
\end{equation*}
Then $\depth(R/I)\geq\min\{\depth(R/(I\colon x^k)),\depth(R/(I,x^k))\}$. As $(I\colon x^k)$ and $(I,x^k)$ satisfy the induction hypothesis, and $\ass(I)=\ass(I\colon x^k)\cup\ass(I,x^k)$, we derive 
\[
\depth(R/I)\geq\min\{\depth(R/\fp)\colon \; \fp\in\ass(I)\},
\]
which implies the result together with \cite[Proposition 1.2.13]{BH}. 
\end{proof}

\begin{cor} \label{ass-decomposable ideal satisfies in stanley conjecture}
If $I$ is an unmixed ass-decomposable ideal then 
$$\sdepth(R/I)\geq\depth(R/I).$$
\end{cor}
\begin{proof}
	 Since $I$ is unmixed, $\depth(R/I)=\depth(R/(I\colon x^k)=\depth(R/(I,x^k))$, by Theorem~\ref{thm:depth}. Now using induction on $r$ and the fact from \cite[Lemma~2.2]{Rauf-2010} that 
	\[\sdepth(R/I) \geq \min \{ \sdepth(R/(I\colon x^k)),\, \sdepth(R/(I,x^k))\},\]
	we get the desired inequality.
\end{proof}

As Example~\ref{sep-not-seq} shows, for an ass-decomposable monomial ideal $I$, the quotient ring $R/I$ is not necessarily sequentially Cohen-Macaulay. However, an unmixed ass-decomposable monomial ideal is Cohen-Macaulay.
\begin{cor} \label{unmixed ass-decomposable is CM}
	Let $I$ be an unmixed ass-decomposable monomial ideal. Then, $R/I$ is Cohen-Macaulay. 
\end{cor} 
\begin{proof}
	Since $I$ is unmixed, $\dim(R/I)=\dim (R/\fp)$ for all $\fp\in\ass(I)$. Now, Theorem~\ref{thm:depth} implies the result.
\end{proof}

In the rest of this section, we study the Castelnuovo-Mumford regularity of squarefree ass-decomposable monomial ideals i.e. the Stanley-Reisner ideals of vertex decomposable simplicial complexes.
Let $I  \subseteq S$ be a monomial ideal and
\[  \cdots \to F_2 \to F_1 \to F_0 \to I \to 0 \]
be  graded minimal free resolution of $I$ with $F_i = \oplus_j S(-j)^{\beta^{\mathbb{K}}_{i,j}(I)}$, for all $i$. The numbers $\beta_{i,j}^{\mathbb{K}}(I) = \dim_{\mathbb{K}} \mbox{Tor}^S_i(I, \mathbb{K})_j$ are called the \textit{graded Betti numbers}
of $I$ and
the \textit{Castelnuovo-Mumford regularity} of $I \neq 0$, $\mathrm{reg}(I)$, is given by
\[ \mbox{reg}(I) = \sup\{j - i \colon \quad \beta_{i,j}^{\mathbb{K}}(I) \neq 0\}.\]
Note that $\reg(R/I)=\reg(I)-1$. We say that $I$ has a \textit{$d$-linear resolution} if $I$ is generated by monomials of degree $d=\reg(I)$.

\begin{prop}\label{depth-reg}
 Let $I=\cap_{i=1}^r \fp_i$ be a minimal prime  decomposition of 
 the squarefree monomial ideal $I$. If $I$ is ass-decomposable then $\reg(R/I)\leq r-1$.
\end{prop}
\begin{proof}
We proceed by induction on $r$ to prove the assertion. If $r=1$, we have nothing to prove. Assume that $r>1$ and the assertion holds for all squarefree ass-decomposable monomial ideals with less than $r$ irreducible components. Let $x$ be a decomposing variable of $I$ and consider the short exact sequence
\begin{equation*}
0 \longrightarrow R/(I\colon x)(-1) \longrightarrow R/I\longrightarrow R/(I,x)\longrightarrow 0.
\end{equation*}
Using the above exact sequence along with \cite[Corollary 20.19]{Ei}, we get 
\begin{align*}
\reg(R/I)& \leq\max\{\reg(R/(I\colon x))+1,\,\reg(R/(I,x))\}\\
& \leq (r-2)+1=r-1.
\end{align*}
The last inequality follows from induction hypothesis. 

\end{proof}

If $I= \cap_{i=1}^r \fp_i$ is a minimal prime decomposition of a squarefree monomial ideal $I$, then $r$ is said to be the \textit{index of irreducibility} of $I$. Let $\cG(I)$ denote the minimal set of monomial generators of $I$ and $d(I)=\max\{\deg(u) \ ; \ u\in\cG(I)\}$. It is immediate consequence of the definition of regularity that $\reg(I) \geq d(I)$.	
	
\begin{prop} \label{when reg=d(I)-1}
 Let $I$ be a squarefree ass-decomposable monomial ideal with decomposing variable $x$ and the index of irreducibility $r$.  
    Then the following statements are equivalent.
\begin{enumerate}
	\item[(i)]$\reg(R/I)=r-1$.
	\item[(ii)]$d(I)=r$.
\end{enumerate} 
If the above conditions hold, then  $I=\fp+xJ$ where $\fp$ is a monomial prime ideal and $J$ is  ass-decomposable monomial ideal. 
\end{prop}
\begin{proof}
If (ii) holds, then (i)  follows from  Proposition~\ref{depth-reg} and the obvious relations $\reg(R/I)=\reg(I)-1 \geq d(I)-1$. Now, assume that (i) holds. We use induction on $r$ to obtain (ii). 
If $r=1$, there is nothing to prove. Assume that $r>1$ and the assertion holds for any squarefree ass-decomposable monomial ideal $J$ whose regularity of $R/J$ is less than $r-1$. Let $I=\cap_{i=1}^r \fp_i$ be a  minimal prime  decomposition of $I$ such that $x\notin\cup^s_{i=1}\fp_i$ and $x\in\cap^r_{i=s+1}\fp_i$, for some $1 \leq s <r$.  

By \cite[Lemma 2.10]{Dao-Huneke-Schweig}, we have
\begin{align*}
\reg(R/I)\in\{\reg(R/(I\colon x))+1,\reg(R/(I,x))\}.
\end{align*}
As $\reg(R/(I,x))\leq r-2$, by Proposition~\ref{depth-reg}, we derive 
\[
r-1=\reg(R/I)=\reg(R/(I\colon x))+1.
\]
Thus $\reg(R/(I\colon x))=r-2$. Since $I \colon x$ is a squarefree ass-decomposable monomial ideal, it follows from induction hypothesis that $d(I \colon x)=r-1$. On the other hand, proposition~\ref{depth-reg} implies that $r-2 \leq s-1$, and hence $s=r-1$. Let $\fp$ be the prime ideal generated by $\cG(\fp_r)\setminus\{x\}$. Since  $\fp_r\subseteq\fp_i+(x)$, for $i=1,\dots,r-1$, we get $\fp\subseteq\cap^{r-1}_{i=1}\fp_i$. Therefore, $I \colon x=\cap^{r-1}_{i=1} \fp_i=\fp+J$ for a monomial ideal $J$ whose support does not contain any $y \in \cG(\fp_r)$. Thus
\begin{equation}\label{local eq 1}
I=(\fp+J)\cap(\fp,x)=\fp+xJ.
\end{equation}
Note that $r-1= d(I \colon x) =d(J)$, hence $d(I)=d(J)+1=(r-1)+1=r$, in view of \eqref{local eq 1}. Furthermore, $J$ is ass-decomposable because $\fp+J=\cap^{r-1}_{i=1} \fp_i=I\colon x$ is ass-decomposable. 
\end{proof}

As a consequence of Propositions \ref{depth-reg} and \ref{when reg=d(I)-1}, we have the following result for the edge ideal of chordal graphs. First note that the independence complex of a chordal graph is vertex decomposable (see {\cite[Corollary 7]{Woodroofe} or \cite[Theorem 4.1]{Dochtermann}}) hence the edge ideal of chordal graph is squarefree ass-decomposable ideal. Thus applying our results about (squarefree) ass-decomposable ideals, we obtain the following corollary. Statement (i) of this corollary can be found in \cite[Corollary 5.6]{DS} where it is proved by a different method. Recall that a \textit{star} graph with center $w$ is a graph $G=(V,E)$ where $V=\{w, v_1, \ldots, v_n\}$ and $E=\{\{w,v_i\} \colon \; i=1, \ldots,n\}$. 

\begin{cor}
Let $I=I(G)\subseteq R$ be the edge ideal of a  chordal graph $G$ and $r$ be the index of irreducibility of $I$. Then
\begin{itemize}
\item[\rm (i)] $\depth(R/I)=\min\{\depth(R/\fp) \colon\; \ \fp\in\ass(I)\}=n-\bigh(I)$.
\item[\rm (ii)] $\reg(R/I)\leq r-1$.
\item[(iii)] The followings are equivalent:
\begin{enumerate}
\item $\reg(R/I)=r-1$;
\item $r=2$;
\item $G$ is a star graph.
\end{enumerate}
If one of the above conditions holds, then $I$ has a linear resolution.
\end{itemize} 
\end{cor}
	
\begin{proof}
As it is mentioned in the beginning of this corollary, the ideal $I$ is a squarefree ass-decomposable ideal. Hence (i) follows form Theorem~\ref{thm:depth} while (ii) follows from Proposition~\ref{depth-reg}. To show that the statements in (iii) are equivalent, first note that (1) and (2) are equivalent by virtue of Proposition~\ref{when reg=d(I)-1}. If (2) holds, then applying Proposition~\ref{when reg=d(I)-1} once again, we observe that $I=xJ$ where $J$ is a squarefree ass-decomposable ideal. Thus $J$ is a prime monomial ideal and consequently $I$ is the edge ideal of a star graph. Conversely, if $G$ is a star graph with root vertex $\{v\}$, then $I=(x_vx_u\ ; \ u\in V(G)\setminus \{v\})=(x_v)\cap\fp$, where $\fp$ is the prime ideal generated by variables $\{x_u \colon \; u\in V(G)\setminus \{v\}\}$. In particular, the index of irreducibility of $I$ is exactly $2$.
\end{proof}



\begin{thebibliography}{CD}
\bibitem{Biermann1}
J. Biermann, C. A. Francisco; H. T. H\'a and A.  Van Tuyl, Partial coloring, vertex decomposability, and sequentially Cohen-Macaulay simplicial complexes. \textit{J. Commut. Algebra} \textbf{7} (2015), no. 3, 337--352. 

\bibitem{Biermann 2}
J. Biermann and A.  Van Tuyl, Balanced vertex decomposable simplicial complexes and their h-vectors. \textit{Electron. J. Combin.} \textbf{20} (2013), no. 3, Paper 15, 12 pp.

\bibitem{Bjorner}
A. Bj\"orner and M. L. Wachs, Shellable nonpure complexes and posets. I, \textit{Trans. Amer. Math. Soc.} \textbf{348} (1996), no. \textbf{4}, 1299--1327.

\bibitem{Brodmann-Sharp}
M. P. Brodmann and R. Y. Sharp, \textit{Local cohomology: an algebraic introduction with geometric applications}, Cambridge Studies in Advanced Mathematics, \textbf{60}. Cambridge University Press, Cambridge, 1998.

\bibitem{BH}
W.~Bruns and J.~Herzog, \textit{Cohen-Macaulay rings}. Cambridge Studies in Advanced Mathematics, \textbf{39}. Cambridge University Press, Cambridge, 1993.

\bibitem{Coleman}
M. Coleman, A. Dochtermann, N. Geist, and S. Oh. Completing and extending shellings
of vertex decomposable complexes. 2020. \href{https://arxiv.org/abs/2011.12225}{arXiv:2011.12225}.

\bibitem{Cook-Nagel}
D. Cook II and U. Nagel, Cohen-Macaulay graphs and face vectors of flag complexes. \textit{SIAM J.
Discrete Math.} \textbf{26} (2012), no. 1, 89--101.

\bibitem{Dao-Huneke-Schweig}
H. L. Dao, C. Huneke and J. Schweig, Bounds on the regularity and projective dimension
of ideals associated to graphs. \textit{J. Algebraic Combin}. \textbf{38} (2013), no. 1, 37--55.

\bibitem{DS}
H. L. Dao and J.~Schweig, Projective dimension, graph domination parameters, and independence complex homology. \textit{J. Combin. Theory Ser. A} \textbf{120} (2013), no. \textbf{2}, 453--469.


\bibitem{Dochtermann}
A. Dochtermann and A. Engstr\"om, Algebraic properties of edge ideals via combinatorial topology.
\textit{Electron. J. Combin.} \textbf{16} (2009), no. 2, Special volume in honer of Anders Bj\"orner, Research Paper 2, 24 pp. 

\bibitem{Duval 2}
A. M. Duval,  Algebraic shifting and sequentially Cohen-Macaulay simplicial complexes. \textit{Electron. J. Combin.} \textbf{3} (1996), no. 1, Research Paper 21, approx. 14 pp.

\bibitem{Duval}
A. M. Duval,  A non-partitionable Cohen–Macaulay simplicial complex. \textit{Adv. Math.}, \textbf{299} (2016), 381--395.

\bibitem{Ei}
D. Eisenbud, \textit{Commutative Algebra, with a View toward Algebraic Geometry}, Springer-Verlag, Berlin-Heidelberg-New York, (1995).

\bibitem{Francisco-Van Tuyl}
C. A. Francisco and A. Van Tuyl. Sequentially Cohen-Macaulay edge ideals. \textit{Proc. Amer. Math. Soc.}, \textbf{135}(8)(2007),2327--2337.

\bibitem{survey}
J. Herzog, \textit{A survey on Stanley depth}, in  Monomial ideals, computations and applications, 3--45, Lecture Notes in Math. (2013), Springer, Heidelberg.

\bibitem{Herzog-Popescu}
J. Herzog and D. Popescu, Finite filtrations of modules and shellable multicomplexes. \textit{Manuscripta Math.} \textbf{121} (2006), no. 3, 385--410.

\bibitem{JS} 
R. Jafari, and H. Sabzrou, Associated radical ideals of monomial ideals. \textit{Comm. Algebra} \textbf{47} (2019), no. 3, 1029--1042.


\bibitem{KM}
F. Khosh-Ahang and S. Moradi, Regularity and projective dimension of the edge ideal of $C_5$-free vertex decomposable graphs. \textit{Proc. Amer. Math. Soc.} \textbf{142} (2014), no. 5, 1567--1576. 

\bibitem{FM}
F. Mohammadi, Powers of the vertex cover ideal of a chordal graph, \textit{Comm. Algebra} \textbf{39} (2011), no. \textbf{10}, 3753--3764.

\bibitem{MK}
S. Moradi and F. Khosh-Ahang, On vertex decomposable simplicial complexes and their Alexander duals. \textit{Math. Scand.}  \textbf{118} (2016), no. 1, 43--56.

\bibitem{Moradi-Kiani}
S. Moradi and D. Kiani, Bounds for the regularity of edge ideals of vertex decomposable and
shellable graphs. \textit{Bull. Iranian Math. Soc.} \textbf{36} (2010), no. 2, 267--277.


\bibitem{Provan-Billera}
J. S. Provan and L. J. Billera, Decompositions of simplicial complexes related to diameters
of convex polyhedra. \textit{Math. Oper. Res.} \textbf{5} (1980), no. 4, 576--594.

\bibitem{RY}
R. Rahmati-Asghar and S. Yassemi, $k$-decomposable  monomial  ideals, \textit{Algebr. Colloq.} \textbf{22} (Spec 1) (2015), 745--756.

\bibitem{Rauf-2010}
A. Rauf, Depth and Stanley depth of multigraded modules. \textit{Comm. Algebra} {\bf 38} (2010), no.~2, 773--784.


\bibitem{Schenzel}
P. Schenzel, \textit{On the dimension filtration and Cohen-Macaulay filtered modules}, Proceed. of
the Ferrara meeting in honour of Mario Fiorentini, ed. F. Van Oystaeyen, Marcel Dekker,
New-York, 1999.

\bibitem{Stanley-1982}
R. P. Stanley, Linear Diophantine equations and local cohomology, \textit{Invent. Math.} {\bf 68} (1982), no.~2, 175--193.

\bibitem{Stanley-book}
R. P. Stanley, \textit{Combinatorics and Commutative Algebra}, 2nd Edition, Progress in Mathematics, 41, Birkh\"auser Boston, Inc., Boston, MA, 1996.

\bibitem{Van Tuyl}
A. Van Tuyl, Sequentially Cohen-Macaulay bipartite graphs: vertex decomposability and
regularity. \textit{Arch. Math.} (Basel) \textbf{93} (2009), no. 5, 451--459.

\bibitem{Wachs}
M. L. Wachs, Obstructions to shellability. \textit{Discrete Comput. Geom.} \textbf{22} (1999), no. 1, 95--103.

\bibitem{Woodroofe}
R. Woodroofe, Vertex decomposable graphs and obstruction to shellability. \textit{Proc. Amer. Math. Soc.} \textbf{137}  (2009), no. 10,  3235--3246.

\bibitem{Woodroofe2} 
R. Woodroofe, Chordal and sequentially Cohen-Macaulay clutters. \textit{Electron.~J.~Combin.} {\textbf 18} (2011), no. 1, Paper 208.

\end{thebibliography}
\end{document}